\documentclass[12pt,a4paper,oneside,reqno]{article}

\usepackage[utf8]{inputenc}
\usepackage[english]{babel}
\usepackage[left=1in,right=1in,top=1in,bottom=1in]{geometry}
\usepackage{amsfonts,amssymb,amsmath,amsthm}
\usepackage[symbol]{footmisc}
\usepackage{longtable,array}
\newcolumntype{L}[1]{>{\raggedright\arraybackslash$}p{#1}<{$}}
\usepackage{enumerate}
\usepackage{cite}
\usepackage{xcolor}

\usepackage{hyperref}
\hypersetup{hidelinks}
\usepackage{fancyhdr}
\theoremstyle{plain}
\newtheorem{theorem}{Theorem}
\newtheorem{lemma}[theorem]{Lemma}

\newtheorem{corollary}[theorem]{Corollary}
\theoremstyle{definition}
\newtheorem{remark}[theorem]{Remark}
\newtheorem{definition}[theorem]{Definition}

\makeatletter
\@removefromreset{theorem}{section}\@removefromreset{equation}{section}

\makeatother
\usepackage{adjustbox}
\usepackage{tikz}\usetikzlibrary{arrows}

\IfFileExists{fouriernc.sty}{\usepackage{fouriernc}}{}

\newcommand{\doi}[1]{\href{https://doi.org/#1}{doi:\detokenize{#1}}}
\newcommand{\arxiv}[1]{\href{https://arxiv.org/abs/#1}{arXiv:#1}}
\newcommand{\arxivsearch}[1]{\href{https://arxiv.org/search/?query=#1}{arXiv}}

\usepackage{stmaryrd}

\providecommand\HS[1]{\leavevmode\null\hspace{#1mm}}

\newcounter{ITEM}

\providecommand\mL{\mathfrak{L}}

\providecommand\wdots{, ...\HS{0.2}, }
\providecommand\xx{x}

\begin{document}

\noindent\fbox{\parbox{\dimexpr\textwidth-2\fboxsep-2\fboxrule\relax}{\small
\textbf{Corrected version (verified by Claude).} This is a corrected version of the paper I. Kaygorodov, V. Lopatkin, Z. Zhang,
\emph{Transposed Poisson structures on Galilean and solvable Lie algebras},
Journal of Geometry and Physics, 187 (2023), 104781,
\doi{10.1016/j.geomphys.2023.104781}.}}

\bigskip

\begin{center}
{\Large Transposed Poisson structures on\\ Galilean and solvable Lie algebras}

\bigskip

{\bf Ivan Kaygorodov, Viktor Lopatkin, and Zerui Zhang}
\end{center}

\medskip

\noindent{\small\textbf{Abstract.} Transposed Poisson structures on complex Galilean type Lie algebras and superalgebras are described.
It is proven that the principal Galilean Lie algebras and superalgebras do not admit nontrivial transposed Poisson structures.
We also prove that each complex finite-dimensional solvable Lie algebra of dimension greater than one admits a nontrivial $\frac{1}{2}$-derivation and a nontrivial ${\rm Hom}$-Lie structure, and that each complex finite-dimensional non-abelian solvable Lie algebra admits a nontrivial transposed Poisson structure.}

\bigskip


\bigskip


\bigskip

\tableofcontents

\bigskip

\section*{Introduction}\addcontentsline{toc}{section}{Introduction}
Poisson algebras arose from the study of Poisson geometry in the 1970s and have appeared in an extremely wide range of areas in mathematics and physics, such as Poisson manifolds, algebraic geometry, operads, quantization theory, quantum groups, and classical and quantum mechanics. The study of all possible Poisson algebra structures with a certain Lie or associative part is an important problem in the theory of Poisson algebras \cite{jawo,said2,YYZ07,kk21}.
Recently, Bai, Bai, Guo, and Wu \cite{bai20} introduced a dual notion of the Poisson algebra, called a transposed Poisson algebra, by exchanging the roles of the two binary operations in the Leibniz rule defining a Poisson algebra.
They showed that transposed Poisson algebras not only share common properties of Poisson algebras, including the closedness under tensor products and the Koszul self-duality as an operad, but also admit a rich class of identities. Moreover, a transposed Poisson algebra naturally arises from a Novikov--Poisson algebra by taking the commutator Lie algebra of the Novikov algebra.
Later, Ferreira, Kaygorodov, and Lopatkin established a relation between $\frac{1}{2}$-derivations of Lie algebras and
transposed Poisson algebras \cite{FKL}.
These ideas were used to describe all transposed Poisson structures
on the Witt algebra, the Virasoro algebra, and the algebras $\mathcal{W}(a,b)$ \cite{FKL};
on the twisted Heisenberg--Virasoro algebra, the Schr\"odinger--Virasoro algebra, and the extended Schr\"odinger--Virasoro algebra \cite{yh21};
and on Block Lie algebras and superalgebras \cite{kk22}.
 The algebraic and geometric classification of $3$-dimensional transposed Poisson algebras is given in \cite{bfk23}.	

Galilei groups and their Lie algebras are important objects in theoretical physics and attract a lot of attention in related mathematical areas, see for example \cite{bg09,mpv16,bg091,s10,csm19,gm13,aks13  ,MT10,lmz14,tang19,GLP16,DT19,AS16,xs20}.
The present paper is dedicated to the study of transposed Poisson structures on
various Galilean type Lie algebras and superalgebras.
The last section of the paper is devoted to $\frac{1}{2}$-derivations of Lie algebras.
Namely, we prove that each complex finite-dimensional solvable Lie algebra of dimension greater than one admits a nontrivial $\frac{1}{2}$-derivation, and that each complex finite-dimensional non-abelian solvable Lie algebra admits a nontrivial transposed Poisson structure.

\section{Preliminaries}

The study of $\delta$-derivations of Lie algebras was initiated by Filippov in 1998 \cite{fil1}. The space of $\delta$-derivations includes derivations, antiderivations, and elements of the centroid.
During the last 20 years, $\delta$-derivations of prime Lie algebras and
$\delta$-superderivations of simple Lie and Jordan superalgebras
have been investigated (see \cite{zusma,k12} and references therein).

\begin{definition}\label{deltadif}
 Let $\mathfrak{L}$ be a   superalgebra and $\delta$ an element of the ground field. A homogeneous linear map $\varphi\colon\mathfrak{L}\to\mathfrak{L}$ is called a $\delta$-superderivation if
 \begin{center}
  $\varphi[a,b] = \delta \left([\varphi(a),b] + (-1)^{\mathrm{deg}(a)\mathrm{deg}(\varphi)}[a,\varphi(b)]\right).$
 \end{center}
\end{definition}

The main example of $\frac{1}{2}$-derivations is the multiplication by an element of the ground field;
we call such $\frac{1}{2}$-derivations trivial.
The space of all $\frac{1}{2}$-(super)derivations of $\mathfrak L$ is denoted by $\Delta(\mathfrak L).$

\begin{lemma}\label{[D1,D2]}
 Let $\varphi_1$, $\varphi_2$ be  $\delta_1$- and  $\delta_2$-superderivations of a superalgebra. Then the supercommutator
 \begin{center}
    $\llbracket \varphi_1,\varphi_2\rrbracket_s= \varphi_1\varphi_2-(-1)^{deg(\varphi_1)deg(\varphi_2)}\varphi_2\varphi_1$
 \end{center}is a   $\delta_1\delta_2$-superderivation.
Similarly,
the commutator $\llbracket \varphi_1,\varphi_2\rrbracket$ of $\delta_1$- and  $\delta_2$-derivations of an algebra is a   $\delta_1\delta_2$-derivation.
\end{lemma}

The definition of   transposed Poisson algebras was given in a paper by Bai, Bai, Guo, and Wu \cite{bai20}.

\begin{definition}\label{tpa}
Let ${\mathfrak L}$ be a vector space equipped with two nonzero bilinear operations $\cdot$ and  $[\cdot, \cdot]$.
The triple  $({\mathfrak L},\cdot,[\cdot,\cdot])$  is called a transposed Poisson algebra if $({\mathfrak L},\cdot)$ is a commutative associative algebra and
 $({\mathfrak L},[\cdot,\cdot])$   is a Lie algebra that satisfies the following compatibility condition
\begin{equation}\label{link1}
2z\cdot [x,y]=[z\cdot x,y]+[x,z\cdot y].\end{equation}
\end{definition}

Summarizing Definitions \ref{deltadif} and  \ref{tpa}  we have the following key lemma.
    \begin{lemma}\label{glavlem}

Let $({\mathfrak L},\cdot,[\cdot,\cdot])$   be a transposed Poisson algebra
and $z$ an arbitrary element from ${\mathfrak L}.$
Then the right multiplication $R_z$ in the associative commutative algebra $({\mathfrak L},\cdot)$ gives a $\frac{1}{2}$-derivation of the Lie algebra  $({\mathfrak L}, [\cdot,\cdot])$.
\end{lemma}

By \cite{FKL}, we have the following results.

\begin{theorem}\label{princth}
Let ${\mathfrak L}$ be a Lie algebra (or superalgebra) of dimension~$>1$  without non-trivial $\frac{1}{2}$-derivations.
Then every transposed Poisson structure defined on ${\mathfrak L}$ is trivial.
\end{theorem}

\begin{definition}
The Witt algebra   is spanned by generators $\{ L_n \}_{ n \in \mathbb{Z}}$. These generators satisfy
 $$
 [L_m,L_n] = (m-n)L_{m+n}.
 $$
\end{definition}

\begin{theorem}\label{teo_witt}
Let $\varphi$ be a $\frac{1}{2}$-derivation of the Witt algebra ${\mathfrak L}.$
Then there is a finite set $\{ \alpha_i \}_{i \in {\mathbb Z}}$ of elements of the ground field such that $\varphi (e_i)=\sum\limits_{j \in \mathbb{Z}} \alpha_{j} e_{i+j}.$
Every finite set $\{ \alpha_i \}_{i \in {\mathbb Z}}$ of elements of the ground field gives a $\frac{1}{2}$-derivation of ${\mathfrak L}$ in this way.
\end{theorem}

\begin{definition}
The Virasoro algebra  is spanned by generators $\{ L_n \}_{ n \in \mathbb{Z}}$ and the central element $c$. These generators satisfy
 $$
 [L_m,L_n] = (m-n)L_{m+n} + (m^3 - m)\delta_{m+n,0}c.
 $$
\end{definition}

\begin{theorem}\label{virasoro}
There are no non-trivial $\frac{1}{2}$-derivations of the Virasoro algebra.
\end{theorem}

All algebras and superalgebras are considered over the complex field.

\section{TP-structures
on Galilean algebras}

\begin{definition}For every integer~$d\geq 3$, the
  Lie algebra $\mathfrak{gal}(d)$ of the Galilean  group
 (it seems that it first appeared in \cite{Bargmann})
 is   generated by the following relations:
 \begin{longtable}{lcl}
  $[J_{i,j}, J_{p,q}]$&$ = $&$ \delta_{i,p} J_{j,q} - \delta_{i,q}J_{j,p} - \delta_{j,p} J_{i,q} + \delta_{j,q}J_{i,p} $ \\
  $[J_{i,j}, P_k ] $&$= $&$\delta_{i,k}P_j - \delta_{j,k}P_i$ \\
    $[J_{i,j}, C_k] $&$= $&$ \delta_{i,k}C_j - \delta_{j,k}C_i$ \\
    $ [C_i, H]$&$ = $&$P_i,$
 \end{longtable}
\noindent 
where
$1 \le i, j,k,p,q \le d$ and $i\ne j, p\ne q$
and $J_{i,j}$ are   antisymmetric tensors
 (namely, we have $J_{i,i}=0$ and~$J_{i,j}=-J_{j,i}$).
 \end{definition}

\begin{theorem}\label{galthm}
There are no nontrivial transposed Poisson structures defined on $\mathfrak{gal}(d).$
\end{theorem}

\begin{proof}
It is clear that $\mathfrak{gal}(d)$ is a $\mathbb{Z}_2$-graded algebra: $\mathfrak{gal}(d) = (\mathfrak{gal}(d))_0 \oplus (\mathfrak{gal}(d))_1$, where $(\mathfrak{gal}(d))_0$ is the direct sum of the algebra $\mathfrak{so}_d$ spanned by all $J_{i,j}$ and
the one-dimensional algebra spanned by $H,$ and
$(\mathfrak{gal}(d))_1$ is spanned by all $P_k,C_k$.
Hence, $\Delta(\mathfrak{gal}(d))$ is also $\mathbb{Z}_2$-graded, i.e., every $\frac{1}{2}$-derivation of $\mathfrak{gal}(d)$ is the sum of an even $\frac{1}{2}$-derivation and an odd one.
Denote by $V=\langle P_k, C_k\rangle_{1\le k\le d}$ the odd part; it is the direct sum of two copies of the natural $\mathfrak{so}_d$-module.

Let $\varphi_0$ be an even $\frac{1}{2}$-derivation.
For pairwise distinct $i,j,k$, we have \begin{center}
$2\varphi_0(J_{j,k})=2\varphi_0[J_{i,j},J_{i,k}]=[\varphi_0(J_{i,j}),J_{i,k}]+[J_{i,j},\varphi_0(J_{i,k})]\in\mathfrak{so}_d$;
\end{center}
hence, $\varphi_0(\mathfrak{so}_d)\subseteq\mathfrak{so}_d$ and $\varphi_0|_{\mathfrak{so}_d}\in\Delta(\mathfrak{so}_d)$.
If $d\neq4$, then $\mathfrak{so}_d$ is simple, and $\varphi_0(x)=\alpha x$ for all $x\in \mathfrak{so}_d$.
If $d=4$, then $\mathfrak{so}_4=\mathfrak{s}_1\oplus\mathfrak{s}_2$ is the direct sum of two simple ideals, and $\varphi_0(x_1+x_2)=\alpha_1x_1+\alpha_2x_2$ for $x_i\in\mathfrak{s}_i$ \cite{FKL}.
Since
\[ 0=2\varphi_0[H,J_{i,j}]= [\varphi_0(H),J_{i,j}]+[H,\varphi_0(J_{i,j})]=[\varphi_0(H),J_{i,j}], \]
the element $\varphi_0(H)$ commutes with $\mathfrak{so}_d$; hence, $\varphi_0(H)=\beta H$ for some $\beta$.
Let $\mathfrak{s}$ be $\mathfrak{so}_d$ (if $d\neq4$) or one of the ideals $\mathfrak{s}_i$ (if $d=4$), and let $\alpha$ be the corresponding scalar.
For $x\in\mathfrak{s}$ and $u\in V$, we have $2\varphi_0([x,u])=\alpha[x,u]+[x,\varphi_0(u)]$, i.e., the map $S=\varphi_0|_V-\alpha\,{\rm id}_V$ satisfies $2S\circ{\rm ad}_x={\rm ad}_x\circ S$ on $V$ for all $x\in\mathfrak{s}$.
Then $4S\circ{\rm ad}_x{\rm ad}_y={\rm ad}_x{\rm ad}_y\circ S$ for all $x,y\in\mathfrak{s}$, whence $4S\circ{\rm ad}_{[x,y]}={\rm ad}_{[x,y]}\circ S=2S\circ{\rm ad}_{[x,y]}$.
Since $\mathfrak{s}=[\mathfrak{s},\mathfrak{s}]$, we obtain $S\circ{\rm ad}_x=0$ and ${\rm ad}_x\circ S=0$ for all $x\in \mathfrak{s}$.
Since $[\mathfrak{s},V]=V$ and $V$ has no nonzero $\mathfrak{s}$-invariant elements, we have $S=0$, i.e., $\varphi_0(u)=\alpha u$ for all $u\in V$.
In particular, $\alpha_1=\alpha_2$ if $d=4$.
Finally,
 \[2\alpha P_i=2\varphi_0(P_i)= 2\varphi_0[C_i, H] = [\varphi_0(C_i), H]+[C_i, \varphi_0(H)]=(\alpha+\beta)P_i,\]
 which gives $\alpha=\beta$, and $\varphi_0$ is trivial.

Let $\varphi_1$ be an odd $\frac{1}{2}$-derivation. Then $\llbracket\varphi_1,  \mathbf{ad}_{\mathbb{U}_i}\rrbracket$ is an even $\frac{1}{2}$-derivation for $\mathbb{U} \in \{ P, C\}$ by Lemma \ref{[D1,D2]}.
Hence,
$\llbracket \varphi_1,  \mathbf{ad}_{\mathbb{U}_i}\rrbracket=  \alpha_{\mathbb{U}_i} {\rm id}$ and, for $j\neq i$,
\[
\alpha_{\mathbb{U}_i} J_{i,j}= \llbracket \varphi_1,  \mathbf{ad}_{\mathbb{U}_i} \rrbracket(J_{i,j})=
\varphi_1[\mathbb{U}_i, J_{i,j}]-[\mathbb{U}_i, \varphi_1(J_{i,j})]=
-\varphi_1(\mathbb{U}_j).
\]
Since the right-hand side does not depend on $i$, we have $\alpha_{\mathbb{U}_i}=0$ and
$\varphi_1(\mathbb{U}_j)=0.$
Let $\varphi_1(J_{j,k})=\sum_t( \gamma^{j,k}_t P_t+ \beta^{j,k}_t C_t)$. For pairwise distinct $i,j,k$, we have
\begin{center}
$2 \varphi_1(J_{j,k})=
2\varphi_1 [J_{i,j},J_{i,k}]=
[\varphi_1(J_{i,j}),J_{i,k}]+[J_{i,j},\varphi_1(J_{i,k})]=$
$
 -\gamma^{i,j}_iP_k+\gamma^{i,j}_kP_i-\beta^{i,j}_iC_k+\beta^{i,j}_kC_i  +
\gamma^{i,k}_iP_j-\gamma^{i,k}_jP_i+\beta^{i,k}_iC_j-\beta^{i,k}_jC_i,
$\end{center}
which gives $\gamma^{j,k}_t=\beta^{j,k}_t=0$ for $t\notin\{i,j,k\}$,
\[2\gamma^{j,k}_j=\gamma^{i,k}_i, \quad 2\beta^{j,k}_j=\beta^{i,k}_i, \quad 2\gamma^{j,k}_i=\gamma^{i,j}_k-\gamma^{i,k}_j, \quad 2\beta^{j,k}_i=\beta^{i,j}_k-\beta^{i,k}_j.\]
Interchanging $i$ and $j$ in the first two relations, we obtain $\gamma^{j,k}_j=\beta^{j,k}_j=0$ and, since $J_{j,k}=-J_{k,j}$, also $\gamma^{j,k}_k=\beta^{j,k}_k=0$.

Let $d\geq4$. For each $t\notin\{j,k\}$ we can choose $i\notin\{j,k,t\}$; hence, $\gamma^{j,k}_t=\beta^{j,k}_t=0$, i.e., $\varphi_1(J_{j,k})=0$.
From $0=2\varphi_1[J_{j,k},H]=[J_{j,k},\varphi_1(H)]$ it follows that $\varphi_1(H)\in V$ is $\mathfrak{so}_d$-invariant; hence, $\varphi_1(H)=0$.
Summarizing, $\varphi_1=0$, $\Delta(\mathfrak{gal}(d))$ is trivial, and, by Theorem \ref{princth}, there are no nontrivial transposed Poisson structures defined on $\mathfrak{gal}(d)$.

Let $d=3$. Then the relations above give $\gamma^{2,3}_1=-\gamma^{1,3}_2=\gamma^{1,2}_3$ and $\beta^{2,3}_1=-\beta^{1,3}_2=\beta^{1,2}_3$.
From $0=2\varphi_1[J_{j,k},H]=[\varphi_1(J_{j,k}),H]+[J_{j,k},\varphi_1(H)]$ and $[C_t,H]=P_t$, it follows that the $C$-component of $\varphi_1(H)$ is $\mathfrak{so}_3$-invariant, and hence zero, and then that $\beta^{j,k}_i=0$ and $\varphi_1(H)=0$.
Therefore, $\Delta(\mathfrak{gal}(3))$ is spanned by the identity map and the map $\phi$ defined by
\[\phi(J_{2,3})=P_1,\quad \phi(J_{1,3})=-P_2,\quad \phi(J_{1,2})=P_3,\quad \phi(P_k)=\phi(C_k)=\phi(H)=0;\]
in particular, $\mathfrak{gal}(3)$ has a nontrivial $\frac{1}{2}$-derivation, and Theorem \ref{princth} cannot be applied.
Let $(\mathfrak{gal}(3),\cdot,[\cdot,\cdot])$ be a transposed Poisson algebra. By Lemma \ref{glavlem}, there are linear functions $a,b$ such that $z\cdot w=a(z)w+b(z)\phi(w)$ for all $z,w$.
Using the commutativity of $\cdot$ for $z=H$ and $w\in\{P_k,C_k\}$, we obtain $a(H)=a(P_k)=a(C_k)=0$; using it for $z=J_{i,j}$ and $w=H$, we obtain $a(J_{i,j})=0$ and $b(H)=0$.
Now, $z\cdot w=b(z)\phi(w)$; taking $z\in\{P_k,C_k\}$ and $w=J_{1,2}$ (or $w=J_{1,3}$), we obtain $b(P_k)=b(C_k)=0$, and the relation $b(J_{1,2})\phi(J_{1,3})=b(J_{1,3})\phi(J_{1,2})$ and its analogues give $b(J_{i,j})=0$.
Hence, the multiplication $\cdot$ is zero, i.e., there are no nontrivial transposed Poisson structures defined on $\mathfrak{gal}(3)$.
\end{proof}

\section{TP-structures on
infinite extension of Galilean algebras  }

\begin{definition}  For every~$\ell \in \mathbb{Z} + \frac{1}{2}$, the infinite
  extension of Galilean algebra $\mathfrak{G}$ (depending on~$\ell$)
 (it seems that it first appeared in \cite{MT10})
 is generated by the following relations:
 \begin{longtable}{lcl}
     $[L^m,L^n] $&$= $&$(m-n)L^{m+n}$\\
     $[L^m,J^n_{i,j}] $&$=$&$ -n J^{n+m}_{i,j}$ \\
     $[J_{i,j}^m, J_{p,q}^n] $&$=$&$ \delta_{i,p} J_{j,q}^{m+n}  - \delta_{i,q} J_{j,p}^{m+n} - \delta_{j,p} J_{i,q}^{m+n} + \delta_{j,q} J_{i,p}^{m+n} $\\
     $[L^m,P_i^k]$&$ =$&$ (\ell m -k)P_i^{m+k}$\\
     $[J_{i,j}^m, P_t^k]$&$=$&$ \delta_{i,t} P_j^{m+k} - \delta_{j,t}P_i^{m+k}$
 \end{longtable}
\noindent  where
$d\in \mathbb{N}, n,m,  t   \in \mathbb{Z}$,
$k\in \mathbb{Z} + \frac{1}{2}$,  $1 \le i \ne j\le d$,   and $J_{i,j}$ are   antisymmetric tensors.
\end{definition}

\begin{theorem}\label{Gthm}
There are no nontrivial transposed Poisson structures defined on $\mathfrak{G}.$
\end{theorem}

\begin{proof}
We will use the standard way for proving that each transposed Poisson  structure is trivial.
After proving that each $\frac{1}{2}$-derivation of $\mathfrak{G}$ is trivial,
we are applying Theorem \ref{princth} and having that there are no non-trivial transposed Poisson   structures on   $\mathfrak{G}.$

\smallskip

It is clear that $\mathfrak{G}$ is a $\mathbb{Z}_2$-graded algebra:
\begin{center}
    $\mathfrak{G}_0= \langle  L^m, J_{i,j}^n \mid m,n\in \mathbb{Z}, 1\leq i\neq j\leq d \rangle$  and
  $\mathfrak{G}_1=  \langle  P_{t}^k \mid t\in \mathbb{Z}, k\in \mathbb{Z} + \frac{1}{2}  \rangle$.
\end{center}
On the other hand   $\langle  J_{i,j}^n \mid 1\leq i\neq j\leq d, n\in \mathbb{Z}  \rangle$  is isomorphic to $\mathfrak{so}_n \otimes \mathbb{C}[t, t^{-1}]$ and $\langle L^m \mid m\in \mathbb{Z}\rangle$ is isomorphic to the Witt algebra.

Let  $\varphi_0$   be an    even $\frac{1}{2}$-derivation.
It is easy to see, that   \begin{center} $\varphi_0( J_{i,j}^n) \subseteq \langle  J_{i,j}^n \mid 1\leq i\neq j\leq d, n\in \mathbb{Z}  \rangle$.\end{center} Thanks to \cite{zusma}, the description of $\frac{1}{2}$-derivations of
$\mathfrak{so}_n \otimes \mathbb{C}[t, t^{-1}]$ is controlling by the space of $\frac{1}{2}$-derivations of $\mathfrak{so}_n:$
\begin{center}
$\Delta(\mathfrak{so}_n \otimes \mathbb{C}[t, t^{-1}]) \cong \Delta(\mathfrak{so}_n) \otimes \mathbb{C}[t, t^{-1}].$
\end{center}
If $d\geq3$ and $d\neq4$, then $\Delta(\mathfrak{so}_d)$ is trivial.
(If $d=4$, then $\mathfrak{so}_4=\mathfrak{s}_1\oplus\mathfrak{s}_2$ is the direct sum of two simple ideals and $\Delta(\mathfrak{so}_4)$ is spanned by the projections onto $\mathfrak{s}_1$ and $\mathfrak{s}_2$; in this case $\varphi_0((x_1+x_2)\otimes t^n)=\sum_t(\alpha_tx_1+\alpha'_tx_2)\otimes t^{n+t}$ for $x_i\in \mathfrak{s}_i$, and the computation below, applied to each of the two ideals separately, gives $\alpha_t=\alpha'_t=0$ for $t\neq0$ and $\alpha_0=\alpha'_0=\beta_0$. The case $d\leq2$ is considered at the end of the proof.)
Hence, for $d\geq3$, we may assume
\begin{center}
$\varphi_0(J_{i,j}^n) =\sum_t \alpha_{t} J_{i,j}^{n+t}=\sum_t \alpha_{t-n} J_{i,j}^{t}$.
\end{center}  It follows that~$\varphi_0$ induces a $\frac{1}{2}$-derivation on the Witt algebra 
\begin{center}$\langle L^m\mid m\in \mathbb{Z}  \rangle\cong \mathfrak{G}_0/\langle  J_{i,j}^n \mid n\in \mathbb{Z}, 1\leq i\neq j\leq d\rangle$.
\end{center} So we may assume~$\varphi_0(L^m) =\sum_t \beta_{t-m} L^{t} +\sum_{u,v,t}\gamma_{u,v}^{m,t} J_{u,v}^t$. By applying the $\frac{1}{2}$-derivation~$\varphi_0$ on $ -n J^{n+m}_{i,j}=[L^m,J^n_{i,j}]$, we obtain that
{\small \begin{longtable}{l}
 $2(-n)\sum_t \alpha_{t-m} J_{i,j}^{n+t} =2(-n)\sum_t \alpha_{t-m-n} J_{i,j}^{t}$\\
  $=[\sum_t \beta_{t-m} L^{t} +\sum_{u,v,t}\gamma_{u,v}^{m,t} J_{u,v}^t, J_{i,j}^{n}]
  + [L^m, \sum_t \alpha_{t-n} J_{i,j}^{t}]$\\
  $=\sum_t \beta_{t-m}(-n)J_{i,j}^{n+t}
  +\sum_{u,v,t}\gamma_{u,v}^{m,t} (\delta_{u,i}J_{v,j}^{n+t}-\delta_{u,j}J_{v,i}^{n+t}
  -\delta_{v,i}J_{u,j}^{n+t}+\delta_{v,j}J_{u,i}^{n+t})
  +\sum_t \alpha_{t-n}(-t) J_{i,j}^{m+t}$\\
  $=\sum_t \beta_{t-m}(-n)J_{i,j}^{n+t}
  +\sum_{u,v,t}\gamma_{u,v}^{m,t} (\delta_{u,i}J_{v,j}^{n+t}-\delta_{u,j}J_{v,i}^{n+t}
  -\delta_{v,i}J_{u,j}^{n+t}+\delta_{v,j}J_{u,i}^{n+t})   +\sum_t \alpha_{t-m}(-t-n+m) J_{i,j}^{n+t}.$
\end{longtable}}
 
 It follows that
$$\sum_{v\neq i, t}\gamma_{i,v}^{m,t}J_{v,j}^{n+t}
+\sum_{v\neq j, t}\gamma_{j,v}^{m,t}(-J_{v,i}^{n+t})
+\sum_{u\neq i, t}\gamma_{u,i}^{m,t}(-J_{u,j}^{n+t})
+\sum_{u\neq j, t}\gamma_{u,j}^{m,t}(J_{u,i}^{n+t})=0.
$$
 So we obtain~$\gamma_{i,p}^{m,t}=\gamma_{p,i}^{m,t} $ for all~$i\neq p$.
Since~$J_{i,j}=-J_{j,i}$, we obtain~$\varphi_0(L^m) =\sum_t \beta_{t-m} L^{t}$ and thus
$$\sum_t \beta_{t-m}(-n)J_{i,j}^{n+t}+\sum_t \alpha_{t-m}(n+m-t) J_{i,j}^{n+t}=0.$$
It follows that for all fixed $n,m,t$, we have
$$\beta_{t-m}(-n)+\alpha_{t-m}(n+m-t)=0.$$
For~$n\neq 0$ and $t=m$, we deduce~$\alpha_0=\beta_0$; for~$n=0$ and~$t\neq m$, we deduce~$\alpha_p=0$ for all nonzero integer~$p$. It follows that~$\beta_p=0$ for all nonzero integer~$p$. Hence, $\varphi_0$ is trivial on $\mathfrak{G}_0$. Assume~$\varphi_0(x)=\alpha x$ for all~$x\in \mathfrak{G}_0$.

Next, we consider $\varphi_0(P_i^k)=\sum_{  u\in \mathbb{Z}, v\in \mathbb{Z}+\frac{1}{2} } \alpha_{i,k}^{u,v} P_u^v$  and the relation on~$-[L^0, P_i^k]$,  which gives
\begin{center}
    $2k\sum_{ u\in \mathbb{Z}, v\in \mathbb{Z}+\frac{1}{2} } \alpha_{i,k}^{u,v}  P_u^v=2k\varphi_0(P^k_i)=\alpha k P^k_i-[L^0, \varphi_0(P^k_i)]=\alpha k P^k_i+\sum_{  u\in \mathbb{Z}, v\in \mathbb{Z}+\frac{1}{2} }  \alpha_{i,k}^{u,v}v P_u^v.$
\end{center}

So we have~$$(2k-v)\sum_{  u\in \mathbb{Z}, v\in \mathbb{Z}+\frac{1}{2} } \alpha_{i,k}^{u,v}  P_u^v=\alpha k P_i^k.$$
Note that~$2k-v\neq 0$. We deduce~$\alpha_{i,k}^{i,k}=\alpha$ and~$\alpha_{i,k}^{u,v}=0$ if~$(u,v)\neq (i,k)$. Therefore, $\varphi_0$ is trivial.

Let $\varphi_1$ be an odd $\frac{1}{2}$-derivation.    Then
  $\llbracket {\rm ad}_{P^k_i}, \varphi_1 \rrbracket$ gives a $\frac{1}{2}$-derivation,
  which is trivial. Hence,
  \begin{center}$\llbracket \mathbf{ad}_{P^k_i}, \varphi_1 \rrbracket = \alpha_{i,k} {\rm id}.$\end{center}
It is easy to see, that
\begin{center}
    $ \alpha_{i,k} L^m = \llbracket \mathbf{ad}_{P^k_i}, \varphi_1 \rrbracket (L^m)=
[P_i^k, \varphi_1(L^m)]-\varphi_1[P_i^k,L^m]=(\ell m -k) \varphi_1(P_i^{k+m}),$
\end{center}
which gives $\varphi_1(P_i^{k})=0$.
Let us consider
$\varphi_1(J_{i,j}^m).$
Obviously,

\begin{longtable}{lcl}
  $\varphi_1(J_{i,j}^m)$&$=$&$\frac{1}{2}\left([\varphi_1(J_{t,i}^m ), J_{t,j}^0]+[ J_{t,i}^m, \varphi_1(J_{t,j}^0)]\right)$ 
  \\& $\in$ & 
  $\bigcap_{t\neq i, t\neq j}\mathsf{span}
  \left\{P_t^k, P_i^k, P_j^k \mid t\neq i, t\neq j, k\in\mathbb{Z}+\frac{1}{2} \right\}$\\
  &$=$&$\mathsf{span}\left\{ P_i^k, P_j^k \mid   k\in\mathbb{Z}+\frac{1}{2} \right\}.$
\end{longtable}
So we may assume
$\varphi_1(J_{i,j}^m)=\sum_k(\alpha_{i,j}^{m,k}P_i^{k}+\beta_{i,j}^{m,k}P_j^{k})$.
By applying~$\varphi_1$ to $J_{i,j}^{m+n}=[J_{t,i}^{m},J_{t,j}^{n}]$ for~$t\notin \{i,j\}$, we obtain
\[\adjustbox{max width=\linewidth}{$\displaystyle \begin{array}{rll}2\sum_k(\alpha_{i,j}^{m+n,k}P_i^{k}+\beta_{i,j}^{m+n,k}P_j^{k})
&=[\sum_k(\alpha_{t,i}^{m,k}P_t^{k}+\beta_{t,i}^{m,k}P_i^{k}), J_{t,j}^{n}]
+[J_{t,i}^{m},\sum_k(\alpha_{t,j}^{n,k}P_t^{k}+\beta_{t,j}^{n,k}P_j^{k})]&\\
&=-\sum_k\alpha_{t,i}^{m,k}P_j^{n+k}
+\sum_k\alpha_{t,j}^{n,k}P_i^{m+k}&\\
&=-\sum_k\alpha_{t,i}^{m,k-n}P_j^{k}
+\sum_k\alpha_{t,j}^{n,k-m}P_i^{k}.&\end{array}$}\]
So for all fixed pairwise distinct numbers~$i,j,k$, we have~$2\alpha_{i,j}^{m+n,k}=\alpha_{t,j}^{n,k-m}$ and~$2\beta_{i,j}^{m+n,k}
=-\alpha_{t,i}^{m,k-n}$. Let~$m=0$. Then we easily deduce that~$\alpha_{i,j}^{n,k}=0=\beta_{i,j}^{n,k}$.

Noting that
\begin{center}
    $0=-2\varphi_1(J_{i,j}^{n+1})=2\varphi_1[L^n,J_{i,j}^1]=[\varphi_1(L^n),J_{i,j}^1],$
\end{center}
    we have $\varphi_1(L^n)=0$ and  thus we obtain   $\varphi_1=0.$

\smallskip

    Let now $d\leq2$; then $\mathfrak{so}_d$ is abelian, and the above argument has to be modified.
The algebra $\mathfrak{G}$ is graded by $\frac12\mathbb{Z}$ with $\deg L^m=\deg J^m_{i,j}=m$ and $\deg P^k_i=k$; hence, it is enough to consider homogeneous $\frac12$-derivations.
Let $\varphi$ be a homogeneous $\frac12$-derivation of degree $s\in\mathbb{Z}$.
Then $\varphi$ induces a $\frac12$-derivation of the Witt algebra $\mathfrak{G}_0/\langle J^n_{i,j}\rangle$, and, by Theorem \ref{teo_witt}, $\varphi(L^m)=\beta L^{m+s}+\gamma^mJ^{m+s}_{1,2}$ (where $\gamma^m=0$ if $d=1$).
Let $\varphi(P_i^k)=\sum_u a^k_{i,u}P_u^{k+s}$, let $A_k=(a^k_{i,u})$, and let $\mathbf{J}$ be the matrix of the action of $J_{1,2}$ on $\langle P_1,P_2\rangle$.
The relation $[L^m,P^k_i]=(\ell m-k)P^{m+k}_i$ gives
\[2(\ell m-k)A_{m+k}=\beta(\ell(m+s)-k)I+\gamma^m\mathbf{J}+(\ell m-k-s)A_k.\]
For $m=0$, we obtain $A_k=\frac{1}{s-k}\big(\beta(\ell s-k)I+\gamma^0\mathbf{J}\big)$.
Taking $m$ odd and $k=\ell m\in\mathbb{Z}+\frac12$ and comparing the scalar parts, we obtain $\frac{\beta sk(1-\ell)}{s-k}=0$; hence, $\beta=0$ if $s\neq0$.
Comparing the $\mathbf{J}$-parts for $m=1$, we obtain $\gamma^1=\gamma^0\big(\frac{2(\ell-k)}{s-1-k}-\frac{\ell-k-s}{s-k}\big)$ for all $k\in\mathbb{Z}+\frac12$, which is possible only if $\gamma^0=0$; then $\gamma^m=0$ for all $m$.
Hence, $\varphi(L^m)=\beta\delta_{s,0}L^{m}$ and $\varphi(P^k_i)=\beta\delta_{s,0}P_i^{k}$.
If $d=2$, then the relation $[J^n_{1,2},P^k_1]=P_2^{n+k}$ gives $[\varphi(J^n_{1,2})-\beta\delta_{s,0}J^n_{1,2},P^k_1]=0$ for all $k$, which implies $\varphi(J^n_{1,2})=\beta\delta_{s,0}J^n_{1,2}$.
Thus, all even $\frac12$-derivations of $\mathfrak{G}$ are trivial.
Let now $\varphi$ be a homogeneous $\frac12$-derivation of degree $s\in\mathbb{Z}+\frac12$.
As above, $\varphi(P^k_i)=0$; let $\varphi(L^m)=\sum_ib^m_iP_i^{m+s}$ and $\varphi(J^m_{1,2})=\sum_ic^m_iP^{m+s}_i$.
The relation $[L^m,L^n]=(m-n)L^{m+n}$ gives \begin{center}$2(m-n)b^{m+n}_i=(\ell m-n-s)b^n_i-(\ell n-m-s)b^m_i$.
\end{center}
For $n=0$, we obtain $b^m_i=\frac{\ell m-s}{m-s}b^0_i$, and then the case $m=1$, $n=-1$ gives $\frac{2(\ell+2)(\ell-1)}{1-s^2}b^0_i=0$; hence, $b^m_i=0$.
If $d=2$, the relation $[L^m,J^n_{1,2}]=-nJ^{m+n}_{1,2}$ gives $-2nc^{m+n}_i=(\ell m-n-s)c^n_i$; for $n=0$ we obtain $c^0_i=0$, and for $m=0$ we obtain $(n-s)c^n_i=0$, i.e., $c^n_i=0$.
Hence, $\varphi=0$.

    Hence, $\Delta(\mathfrak{G})$ is trivial and there are no nontrivial transposed Poisson structures defined on $\mathfrak{G}$.
\end{proof}

\section{TP-structures on the conformal centrally extended Galilei algebras}

\begin{definition}  For every~$0< \ell \in \mathbb{N}-\frac{1}{2}$, the  
 conformal centrally extended Galilei algebra $\widetilde{\mathfrak{g}}^{(\ell)}$  (it seems that it first appeared in \cite{negro}) is generated by the following relations:
 \begin{longtable}{lcllcllcl}
    $[h,e]$ & $=$ & $2e,$ & $[h,f]$ & $=$ & $-2f,$ & $[e,f]$ & $=$ & $h,$\\
    $[h,p_k]$ & $=$ & $2(\ell-k)p_k,$ & $[e,p_k]$ & $=$ & $kp_{k-1},$ & $[f,p_k]$ & $=$ & $(2 \ell -k)p_{k+1},$ \\
\multicolumn{9}{c}{$[p_k,p_{2 \ell - k}] =  (-1)^{k + \ell + \frac{1}{2}}k! (2 \ell - k)! z,$}
\end{longtable}
\noindent   where  $k$  satisfies that  $0 \le k \le  2\ell$.

\end{definition}

\begin{remark}
$\widetilde{\mathfrak{g}}^{(\frac{1}{2})}$ is the Schrödinger algebra considered in \cite{FKL}.
\end{remark}

\begin{theorem}
 There are no non-trivial transposed Poisson structures on $\widetilde{\mathfrak{g}}^{(\ell)}$.
\end{theorem}
\begin{proof}
We will use the standard way for proving that each transposed Poisson  structure is trivial.
After proving that each $\frac{1}{2}$-derivation of $\widetilde{\mathfrak{g}}^{(\ell)}$ is trivial,
we are applying Theorem \ref{princth} and having that there are no non-trivial transposed Poisson structures on   $\widetilde{\mathfrak{g}}^{(\ell)}.$

\bigskip

It is easy to see that $\widetilde{\mathfrak{g}}^{(\ell)}$ is $\mathbb{Z}_2$-graded, $\widetilde{\mathfrak{g}}^{(\ell)}=(\widetilde{\mathfrak{g}}^{(\ell)})_0 \oplus (\widetilde{\mathfrak{g}}^{(\ell)})_1$, where $(\widetilde{\mathfrak{g}}^{(\ell)})_0$ is generated by $e,f,h,z$, and $(\widetilde{\mathfrak{g}}^{(\ell)})_1$ by all $p_k$. Next, it is clear that $(\widetilde{\mathfrak{g}}^{(\ell)})_0$ is the direct sum of the simple algebra $\mathfrak{sl}_2$ and the one-dimensional algebra generated by $z$.

Let $\varphi_0$   be an even  $\frac{1}{2}$-derivation.
Then it has the following type
$\varphi_0(x) = \alpha x$ for any $x \in \{e,f,h\}$ and
$\varphi_0 (z) = \beta z$.
Next, let $\varphi_0(p_k) = \sum_{t=0}^{2 \ell} \beta_t^{(k)}p_t$.
By
\begin{align*}
   4(\ell-k) \sum_{t=0}^{2 \ell} \beta_t^{(k)}p_t &=  4(\ell - k)\varphi_0(p_k)= 2\varphi_0[h,p_k] =2(\ell-k) \alpha p_k+[h, \varphi_0(p_k)]&\\
  &  =2(\ell-k) \alpha p_k+2(\ell-t) \sum_{t=0}^{2 \ell} \beta_t^{(k)}p_t, &
 \end{align*}
  it follows~$2(\ell-k)\beta_k^{(k)}=2(\ell-k)\alpha$, and for~$t\neq k$, we have~$2(\ell-2k+t)\beta_t^{(k)}=0 $. Since~$\ell\neq k$ and~$\ell\neq 2k-t$, we deduce that~$\varphi_0(p_k) = \alpha p_k$.
It is easy to see, that
\begin{center}
 $\varphi_0(z)= (-1)^{ \ell + \frac{1}{2}} \big((2 \ell)!\big)^{-1}\varphi_0[p_{0},p_{2\ell}]=\alpha z.$
\end{center}
Hence, $\varphi_0$ is trivial.

Let $\varphi_1$ be an odd $\frac{1}{2}$-derivation.
It is clear that $\llbracket \varphi_1,\mathbf{ad}_{p_k} \rrbracket$ is an even $\frac{1}{2}$-derivation
for any $k = 0,\ldots, 2\ell$.
Set
$\varphi_1(x) = \sum_{t=0}^{2\ell} \gamma_t^{(x)} p_t$ for any $x\in \{e,f,h,z\}.$
It is easy to see, that
\begin{center}
$4(\ell-k)\varphi_1(p_k)=2\varphi_1[h, p_k]=[\varphi_1(h), p_k]+[h, \varphi_1(p_k)],$
\end{center}
hence $\varphi_1(p_k) \in \langle e,f,z \rangle$
and, in a  similar way, we can obtain that  $\varphi_1(p_k) \in  \langle z \rangle,$
i.e.
$\varphi_1(p_k) =  \rho^{(k)} z$ for any $k=0,\ldots, 2\ell.$
For any $x\in \{e,f,h,z\}$ we obtain
\begin{longtable}{lcl}
 $\alpha_k f $ & $=$&$ \llbracket\varphi_1, \mathbf{ad}_{p_k}\rrbracket(f) = \varphi_1[p_k,f] - [p_k, \varphi_1(f)]$ \\
 & $=$& $ -(2\ell-k) \rho^{(k+ 1)} z  -
  \gamma_{2\ell-k}^{(f)}  (-1)^{k + \ell + \frac{1}{2}}k! (2\ell - k)! z,$\\
 $\alpha_k e $ & $=$&$ \llbracket\varphi_1, \mathbf{ad}_{p_k}\rrbracket(e) = \varphi_1[p_k,e] - [p_k, \varphi_1(e)]$ \\
 & $=$& $ -k \rho^{(k- 1)} z  -
  \gamma_{2\ell-k}^{(e)}  (-1)^{k + \ell + \frac{1}{2}}k! (2\ell - k)! z,$\\
 $\alpha_k h $ & $=$&$ \llbracket\varphi_1, \mathbf{ad}_{p_k}\rrbracket(h) = \varphi_1[p_k,h] - [p_k, \varphi_1(h)]$ \\
 & $=$& $ -2(\ell-k) \rho^{(k)} z  -
  \gamma_{2\ell-k}^{(h)}  (-1)^{k + \ell + \frac{1}{2}}k! (2\ell - k)! z,$
\end{longtable}
which gives $\alpha_k=0$ and

\begin{longtable}{lcl}
$\gamma^{(f)}_{k}$&$=$&$-\rho^{(2\ell-k+1)} \Big((-1)^{3\ell-k+\frac{1}{2}}  (2\ell-k)! (k-1)!\Big)^{-1},$\\

$\gamma^{(e)}_{k}$&$=$&$-\rho^{(2\ell-k-1)} \Big((-1)^{3\ell-k+\frac{1}{2}}  (2\ell-k-1)! k!\Big)^{-1},$\\

$\gamma^{(h)}_{k}$&$=$&$2(\ell-k)\rho^{(2\ell-k)} \Big((-1)^{3\ell-k+\frac{1}{2}}  (2\ell-k)! k!\Big)^{-1}.$

\end{longtable}

 It follows that
\begin{longtable}{lcl}
 $ 2 \sum_{k=0}^{2\ell} \gamma_k^{(h)} p_k$&$=$&$2\varphi_1(h)=2\varphi_1[e,f]=[\varphi_1(e),f]+[e,\varphi_1(f)]$\\
&$=$&$\sum_{k=0}^{2\ell} \gamma_k^{(e)}(-1)(2\ell-k)p_{k+1}
+\sum_{k=0}^{2\ell} \gamma_k^{(f)}kp_{k-1}$\\
&$=$&$\sum_{k=1}^{2\ell+1} \gamma_{k-1}^{(e)}(-1)(2\ell-k+1)p_{k}
+\sum_{k=-1}^{2\ell-1} \gamma_{k+1}^{(f)}(k+1)p_{k}$\\
&$=$&$\sum_{k=1}^{2\ell} \gamma_{k-1}^{(e)}(-1)(2\ell-k+1)p_{k}
+\sum_{k=0}^{2\ell-1} \gamma_{k+1}^{(f)}(k+1)p_{k}.$
\end{longtable}
So we deduce~$2\gamma_0^{(h)}=\gamma_{1}^{(f)}$, $2\gamma_{2\ell}^{(h)}=-\gamma_{2\ell-1}^{(e)}$ and
$$2\gamma_k^{(h)}=\gamma_{k-1}^{(e)}(-1)(2\ell-k+1)+\gamma_{k+1}^{(f)}(k+1)$$
for~$1\leq k\leq 2\ell-1$.  Combining these with the above formulas on~$\gamma_{k}^{(x)}$ for $x\in \{e,f, h\}$, we deduce  that~$\rho^{(2\ell)}=0$, $\rho^{(0)}=0$, and for~$1\leq k\leq 2\ell-1$, we deduce that~$2(\ell-k)\rho^{(2\ell-k)}=0$; Since~$\ell\neq k$, we obtain~$\rho^{(2\ell-k)}=0$.

It follows that $\varphi_1 =0.$ Hence, $\Delta(\mathfrak{\widetilde{\mathfrak{g}}^{(\ell)}})$ is trivial and there are no non-trivial transposed Poisson structures defined on $\widetilde{\mathfrak{g}}^{(\ell)}.$
\end{proof}

\section{TP-structures on
$\ell$-super  Galilean conformal algebras}

\begin{definition}\label{de17}
  For every~$\ell\in \frac{1}{2}\mathbb N$,  the 
 $\ell$-super Galilean conformal algebra $\mathfrak{gca}(\ell)$
 (it seems that it first appeared in \cite{AS16}) is a Lie superalgebra $\mathfrak{gca}(\ell) = \mathfrak{gca}_0(\ell) \oplus \mathfrak{gca}_1(\ell)$ where $\mathfrak{gca}_0(\ell)$ is generated by all $L_m,P_k$, $c_1$, $c_2$, and $\mathfrak{gca}_1(\ell)$ is generated by all $G_m,H_k$, and the multiplication table is given by the following relations:
 \begin{longtable}{lcl}
     $[L_m,L_n] $&$= $&$(m-n)L_{m+n} + c_1(m^3-m) \delta_{m+n,0}$\\
     $[L_m,P_k] $&$=$&$ (\ell m - k)P_{m+k} + c_2(m^3-m)\delta_{m+k,0}\delta_{\ell,1}$ \\
     $[G_m,G_n] $&$=$&$ 2L_{m+n} + c_1(4m^2-1) \delta_{m+n,0}$\\
     $[G_m,H_k] $&$ =$&$ 2P_{m+k} + c_2(4m^2-1)\delta_{m+k,0}\delta_{\ell,1}$\\
     $[L_m, G_n] $&$=$&$ \left(\frac{m}{2}-n\right) G_{m+n}$ \\
     $[L_m,H_k] $&$=$&$ \left(\frac{2\ell -1}{2}m - k\right)   H_{m+k}   $    \\
     $[P_k, G_m] $&$=$&$ \left(\frac{k}{2} - \ell m\right) H_{k+m},$ \\
 \end{longtable}
where
$m,n\in \mathbb{Z}$ and   $k \in \mathbb{Z}+\ell$.
\end{definition}

  By convention, if~$\ell \neq 1$, then   $\mathfrak{gca}_0(\ell)$  is generated by $\{L_m,P_k,c_1\mid m\in \mathbb{Z}, k\in \mathbb{Z} +\ell\}$.

\begin{remark}
It is clear that  $\mathfrak{gca}_0(\ell)$ is $\mathbb Z_2$-graded and $(\mathfrak{gca}_0(\ell))_0$ isomorphic to the Virasoro algebra,   where $(\mathfrak{gca}_0(\ell))_0$ is generated by $\{L_m,c_1\, |\, m \in \mathbb{Z}\}$.
\end{remark}

\begin{theorem}\label{gcathm}
  There are no nontrivial transposed Poisson superalgebra structures defined on $\mathfrak{gca}(\ell)$.
\end{theorem}

\begin{proof}
  Note   that $\mathfrak{gca}_0(\ell)$ is a $\mathbb{Z}_2$-graded algebra; $\mathfrak{gca}_0(\ell)  = (\mathfrak{gca}_0(\ell))_0   \oplus  (\mathfrak{gca}_0(\ell))_1$, where $(\mathfrak{gca}_0(\ell))_0$ is generated by $\{L_m,c_1\, |\, m \in \mathbb{Z}\}$, and $(\mathfrak{gca}_0(\ell))_1$ by $  \{P_k, c_2   \mid \, k \in \mathbb{Z}+\ell \}$.

Let $\varphi$ be a $\frac{1}{2}$-superderivation of $\mathfrak{gca}(\ell)$. Then  we obtain $\varphi = \varphi_0 + \varphi_1$, and $\varphi_0|_{\mathfrak{gca}_0(\ell)} =  \psi_0 + \psi_1$ is a $\frac{1}{2}$-derivation of $\mathfrak{gca}_0(\ell)$,   where $\psi_0 =(\varphi_0|_{\mathfrak{gca}_0(\ell)})_0$, and $\psi_1 = (\varphi_0|_{\mathfrak{gca}_0(\ell))})_1.$ By Theorem \ref{virasoro}, $\psi_0$ is a trivial $\frac{1}{2}$-derivation of $(\mathfrak{gca}_0(\ell))_0$, say
$\psi_0(L_m) = \varkappa L_m$, $m\in \mathbb{Z}$ and $\psi_0(c_1) = \varkappa c_1.$

To calculate $\psi_0(P_k)$ for any $k\in \mathbb{Z}+\ell$ we set
\begin{center}
$\psi_0(P_k) = \sum_{t\in \mathbb{Z}+\ell}\alpha_t^{(k)}P_t + \rho^{(k)}c_2$ \  and \ $\psi_0(c_2) = \sum_{t \in \mathbb{Z}+\ell} \beta_t P_t + \rho c_2$,    \end{center}
where almost all $\alpha_t^{(k)}, \rho^{(k)}$, $\beta_t$ are zero.

We have
\begin{longtable}{lcl}
    $2\psi_0[L_m, P_k]$ &$=$ & $ [\psi_0(L_m), P_k] +  [L_m, \psi_0(P_k)]$\\
    & $=$ &$   [\varkappa L_m, P_k] +  \sum \alpha_t^{(k)} [L_m, P_t]$ \\
    & $=$ & $ \varkappa (\ell m -k)P_{m+k} +  \varkappa (m^3-m) \delta_{m+k,0}\delta_{\ell,1} c_2$ \\
    & & $+ \sum\alpha_t^{(k)}(\ell m -t) P_{m+t} +   \sum \alpha_t^{(k)} (m^3-m) \delta_{m+t,0}\delta_{\ell,1} c_2.$
\end{longtable}

On the other hand,
\begin{longtable}{lcl}
    $\psi_0[L_m,P_k]$ &$ =$ &$ (\ell m - k) \psi_{0}(P_{m+k}) + (m^3-m) \delta_{m+k,0}\delta_{\ell,1} \psi_0(c_2)$ \\
    &$=$ & $(\ell m - k) \sum \alpha_t^{(m+k)}P_t + (\ell m - k) \rho^{(m+k)} c_2$\\
    & &$ + (m^3-m)\delta_{m+k,0}\delta_{\ell, 1}  \sum \beta_t P_t +  (m^3-m)\delta_{m+k,0}\delta_{\ell, 1}\rho c_2.$
\end{longtable}

It follows that
\[\adjustbox{max width=\linewidth}{$\displaystyle \begin{cases}
   \frac{1}{2} \varkappa (\ell m-k) + \frac{1}{2} (\ell m - k) \alpha_k^{(k)} = (\ell m - k) \alpha_{m+k}^{(m+k)} + (m^3 - m) \delta_{m+k,0}\delta_{\ell, 1} \beta_{m+k}, \\
   \frac{1}{2} (\ell m - t) \alpha_t^{(k)} = (\ell m - k) \alpha_{m+t}^{(m+k)} + (m^3 - m) \delta_{m+k,0}\delta_{\ell, 1} \beta_{m+t},  \qquad t \ne k,\\
   \frac{1}{2} \varkappa (m^3 - m) \delta_{m+k,0}\delta_{\ell,1}  + \frac{1}{2} \alpha_{-m}^{(k)} (m^3 - m) \delta_{\ell,1 }  = (\ell m - k) \rho^{(m+k)} + (m^3-m) \delta_{m+k,0}\delta_{\ell, 1} \rho.
 \end{cases}$}\]

If $m=0$ we then get
 \begin{longtable}{lcl}
  $  \varkappa k +  k \alpha_k^{(k)}$ &$ =$& $2 k \alpha_k^{(k)},$\\
  $ t \alpha_t^{(k)}$ &$ =$&$2 k \alpha_{t}^{(k)}, \qquad t \ne k,$\\
  $k \rho^{(k)} $& $= $& $0,$
 \end{longtable}
\noindent
hence all $\alpha_k^{(k)} = \varkappa$ for $k \ne 0$; all $\alpha_t^{(k)} = 0$ if $t \ne k$ and $t \ne 2k$; and all $\rho^{(k)} = 0$ for $k \ne 0.$  If~$\ell\notin \mathbb{Z}$, then we have~$\alpha_t^{(k)}=0$ for all~$t\neq k$ and~$\alpha_k^{(k)}=\varkappa$ for all~$k$.
If~$\ell\in \mathbb{Z}$, then   for   $ 0\neq$    $t=2k$, we can set $m=1$, then we have $(\ell-2k)\alpha_{2k}^{(k)}=0$. So, for $\ell \ne 2k$, we have $\alpha_{2k}^{(k)}=0$   if~$k\neq 0$.  For the case  $\ell=2k \neq 0  $,  we can take $t=m+2k$, with $m \ne 0$ and $m \ne -k$, then we have $\alpha_{2k}^{(k)}=0$ for all $k \neq 0 $.
  It is easy to see, that if we set $k=0$  and $m\neq 0$,   then the first equality implies that   $\ell \alpha_0^{(0)} = \ell \varkappa$,    then $\alpha_0^{(0)} =  \varkappa$ if $\ell \ne 0.$ Putting $\ell = 0$   and~$1=m=-k$,   by the first equality,    we obtain~$\alpha_0^{(0)} =  \varkappa$,
  therefore all $\alpha_k^{(k)} = \varkappa$, $k \in \mathbb{Z}+\ell.$

     Next, we then get $(m^3-m)\delta_{m+k,0}\delta_{\ell, 1} \beta_{m+k} = 0$, hence $\beta_0 =0$ because  in the case $\ell\ne 1$ the letter $c_2$ is not involved in $\mathfrak{gca}_0(\ell)$ by convention.  Similarly, settting~$m=-k\neq -t$, we deduce~$\beta_t=0$ for all~$t\neq 0$.

   Now we consider the coefficients of~$c_2$. By convention, we have~$\ell=1$. Let~$m=1$. Then the third equality implies that~$(1 - k) \rho^{(1+k)}=0$. It follows that~$\rho^{(t)}=0$ for~$t\neq 2$. Setting~$m=0$ and~$k=2$, we obtain~$\rho^{(2)}=0$. So we deduce~$\rho^{(t)}=0$ for~$t$.   It follows that $\rho = \varkappa$ and this shows that all even $\frac{1}{2}$-derivations of $\mathfrak{gca}_0(\ell)$ are trivial.

 Let~$\psi_1$ be an odd $\frac{1}{2}$-derivation of $\mathfrak{gca}_0(\ell)$. Then for~$P_k$,  the map $\llbracket \psi_1,\mathbf{ad}_{P_k} \rrbracket$ is a trivial even $\frac{1}{2}$-derivation of $\mathfrak{gca}_0(\ell)$.  Assume that $\llbracket \psi_1,\mathbf{ad}_{P_k} \rrbracket=\alpha_{k} {\rm id}$. Suppose that~$\psi_1(P_k)=\sum_{m}\alpha_{k,m}L_m+\rho_kc_1$. Then by~$\llbracket \psi_1,\mathbf{ad}_{P_k} \rrbracket(P_t)=-[P_k,  \psi_1(P_t)]=\alpha_{k}P_t$ for all~$t$ we deduce~$\alpha_{t,m}(\ell m-k)=0$ if~$t\neq k+m$.  Let~$k\neq \ell m$ and~$k\neq t-m$, we obtain~$\alpha_{t,m}=0$ if~$t-m\neq k$. Since~$k$ is arbitrary, we have~$\alpha_{t,m}=0$ for all~$t,m$, and thus~$\psi_1(P_k)=\rho_kc_1$ and~$\llbracket \psi_1,\mathbf{ad}_{P_k} \rrbracket=0$. Since~$[\psi_1(L_m), P_k]=0$, we obtain
\begin{longtable}{lcl}
$0$&$=$&$\frac{1}{2}[\psi_1(L_m), P_k]+\frac{1}{2}[L_m,\psi_1(P_k)]=\psi_1[L_m, P_k]$\\
&$=$&$(\ell m - k)\psi_1(P_{m+k}) + \psi_1(c_2)(m^3-m)\delta_{m+k,0}\delta_{\ell,1}.$\end{longtable}
For~$m=0$ and~$k\neq 0$, we obtain~$\rho_k=0$ for~$k\neq 0$.
 If~$\ell\notin \mathbb{Z}$, then we have~$\psi_1(P_k)=0$ for all~$k$. If~$\ell\in \mathbb{Z}$, then for  $m=1=-k$, we obtain~$(\ell+1)\rho_0=0$. Since~$\ell\neq -1$, we  obtain $\rho_0=0$ and thus~$\psi_1(P_k)=0$ for all~$k$. So~$\psi_1(c_2)(m^3-m)\delta_{m+k,0}\delta_{\ell,1}=0$ and thus~$\psi_1(c_2)=0$.

Now we assume~$\psi_1(L_m)=\sum_k\beta_{m,k}P_k+\rho_m'c_2$. Then for all~$m\neq -n$ or~$m\in \{1,-1,0\}$ or~$n=0$, we have
\begin{align*}
&2(m-n)\big(\sum_k\beta_{m+n,k}P_k+\rho_{m+n}'c_2\big)&\\
= &2(m-n)\psi_1(L_{m+n})=[\psi_1(L_m),L_n]+[L_m, \psi_1(L_n)]&\\
=&-\sum_k\beta_{m,k}[L_n,P_k]+ \sum_k\beta_{n,k}[L_m,P_k]&\\
=&-\sum_k\beta_{m,k}(\ell n - k)P_{n+k}   -c_2\beta_{m,-n}(n^3-n)\delta_{\ell,1}&\\
&+\sum_k\beta_{n,k}(\ell m - k)P_{m+k}
+ c_2\beta_{n,-m}(m^3-m)\delta_{\ell,1}.&
\end{align*}
So for all~$k,n$ and for all~$m$ satisfying~$m\neq -n$ or~$m\in \{1,-1,0\}$ or~$n=0$, we  deduce that
$$2(m-n)\rho_{m+n}'=-\beta_{m,-n}(n^3-n)\delta_{\ell,1}
+\beta_{n,-m}(m^3-m)\delta_{\ell,1}$$
and
$$2(m-n)\beta_{m+n,k}=-\beta_{m,k-n}(\ell n-k+n)+\beta_{n,k-m}(\ell m-k+m). $$
For~$m=1$ and~$n=-1$, we have
$$4\beta_{0,k}=-\beta_{1,k+1}(-\ell-k-1)+\beta_{-1,k-1}(\ell-k+1); $$
For~$m=1$ and~$n=0$, we have
$$2\beta_{1,k}=-\beta_{1,k}(-k)+\beta_{0,k-1}(\ell -k+1); $$
For~$m=0$ and~$n=-1$, we have
$$2\beta_{-1,k}=-\beta_{0,k+1}(-\ell-k-1)+\beta_{-1,k}(-k). $$
Moreover, for~$n=0$, we deduce that~$(2m-k)\beta_{m,k}=(\ell m-k+m)\beta_{0,k-m}$ for all~$m,k$.

If~$\ell\notin \mathbb{N}$, then we deduce
$\beta_{1,k+1}=\frac{\ell-k}{1-k}\beta_{0,k}$, $\beta_{-1,k-1}=\frac{\ell+k}{1+k}\beta_{0,k}$ and thus
$$\big((\ell-k)(\ell+k+1)(1+k)+(\ell+k)(\ell-k+1)(1-k)-4(1-k)(1+k)\big)\beta_{0,k}=0,$$
and from here it  follows that~$(\ell+2)(\ell-1)\beta_{0,k}=0$. So we have~$\beta_{0,k}=0$ and thus~$\beta_{m,k}=0$ for all $m,k$.

If~$1\neq \ell\in \mathbb{N}$, then,  by setting~$k=2m\neq 0$ and~$n=0$, we obtain~$\beta_{0,m}=0$ for all~$m\neq 0$ and thus~$\beta_{m,k}=0$ for all $m,k$.

Now we assume $\ell=1$. Then the map $L_m\mapsto L_m$, $c_1\mapsto c$, $P_m\mapsto L_m\varepsilon$, $c_2\mapsto c\varepsilon$ is an isomorphism of $\mathfrak{gca}_0(1)$ onto the Lie algebra ${\rm Vir}\otimes\mathbb{C}[\varepsilon]/(\varepsilon^2)$, where ${\rm Vir}$ is the Virasoro algebra with the central element $c$.
We already know that $\psi_1(P_k)=\psi_1(c_2)=0$ and that $\psi_1$ maps ${\rm Vir}$ to ${\rm Vir}\,\varepsilon$; let $\psi_1(x)=\theta(x)\varepsilon$ for $x\in{\rm Vir}$.
Since $[x\varepsilon,y]=[x,y]\varepsilon$, the map $\theta$ is a $\frac12$-derivation of ${\rm Vir}$; by Theorem \ref{virasoro}, $\theta=\mu\,{\rm id}$.
Hence, $\psi_1(L_m)=\mu P_m$ and $\psi_1(c_1)=\mu c_2$, and this map is indeed a $\frac12$-derivation of $\mathfrak{gca}_0(1)$.

As a conclusion, if $\ell\neq1$, then $\psi_1=0$ and all $\frac{1}{2}$-derivations of $\mathfrak{gca}_0(\ell)$ are trivial; if $\ell=1$, then each $\frac{1}{2}$-derivation of $\mathfrak{gca}_0(1)$ has the form $\varkappa\,{\rm id}+\mu\Psi_0$, where $\Psi_0(L_m)=P_m$, $\Psi_0(c_1)=c_2$, and $\Psi_0(P_k)=\Psi_0(c_2)=0$.
 
  Now we describe the even $\frac{1}{2}$-derivations $\varphi_0$ of $\mathfrak{gca}(\ell)$.
 If $\ell=1$, then $\mathfrak{gca}(1)\cong\mathfrak{s}\otimes\mathbb{C}[\varepsilon]/(\varepsilon^2)$, where $\mathfrak{s}$ is spanned by $L_m$, $G_m$, $c_1$ (and $P_m=L_m\varepsilon$, $H_m=G_m\varepsilon$, $c_2=c_1\varepsilon$), and the multiplication by $\varepsilon$, i.e., the map $\Psi$ given by
 $\Psi(L_m)=P_m$, $\Psi(G_m)=H_m$, $\Psi(c_1)=c_2$, $\Psi(P_k)=\Psi(H_k)=\Psi(c_2)=0$,
 is an even $\frac12$-superderivation of $\mathfrak{gca}(1)$ extending $\Psi_0$.
 By the above reasoning, replacing $\varphi_0$ by $\varphi_0-\mu\Psi$ if $\ell=1$, we may assume that $\varphi_0(x)=\varkappa x$ for all $x\in \mathfrak{gca}_0(\ell)$.

Suppose that~\begin{center}$\varphi_0(G_m)
=\sum_{k}\mu_{m,k}H_k+\sum_{p}\nu_{m,p}G_p$
and~$\varphi_0(H_k)
=\sum_{t}\mu_{k,t}'H_t+\sum_{p}\nu_{k,p}'G_p$.\end{center}
By applying~$\varphi_0$ to the last relation of Definition~\ref{de17}, we have
$$2\Big(\frac{k}{2}-\ell m\Big)\varphi_0(H_{k+m})
=\varkappa \Big(\frac{k}{2}-\ell m\Big)H_{k+m}+\sum_{p}\nu_{m,p}(\frac{k}{2}-\ell p)H_{k+p}\in \mathsf{span}\{H_k\mid k\in \mathbb{Z}+\ell\}.$$  It follows that~$\varphi_0(H_k)=\sum_{t}\mu_{k,t}'H_t$ and thus
$$2\Big(\frac{k}{2}-\ell m\Big)(\sum_{t}\mu_{k+m,t}'H_t)
=\varkappa \Big(\frac{k}{2}-\ell m\Big)H_{k+m}+\sum_{p}\nu_{m,p}\Big(\frac{k}{2}-\ell p\Big)H_{k+p}.$$
So we have~$2\mu_{k+m,k+m}'=\varkappa +\nu_{m,m}$ for all~$k,m$ satisfying~$k\neq 2\ell m$; and for~$p\neq m$, we obtain
$$2 (\frac{k}{2}-\ell m)\mu_{k+m,k+p}'=\nu_{m,p}(\frac{k}{2}-\ell p).$$
Similarly, by applying $\varphi_0$ on the relation involving~$[L_m,H_k]$, we have
\begin{longtable}{lcl}
  $2(\frac{2\ell-1}{2}m-k)\sum_{t}\mu_{k+m,t}'H_t $&$=$&$2\varphi_0([L_m,H_k])$\\
  &$=$&$\varkappa(\frac{2\ell-1}{2}m-k)H_{k+m}
  +\sum_{t}\mu_{k,t}'[L_m,H_t]$\\
  &$=$&$\varkappa(\frac{2\ell-1}{2}m-k)H_{k+m}
  +\sum_{t}\mu_{k,t}' (\frac{2\ell-1}{2}m-t)H_{t+m}.$
\end{longtable}
So for~$\frac{2\ell-1}{2}m\neq k$, we have
$2\mu_{k+m,k+m}'=\varkappa+\mu_{k,k}'$
and~$$ 2\Big(\frac{2\ell-1}{2}m-k\Big)\mu_{k+m,t+m}'=  \Big(\frac{2\ell-1}{2}m-t\Big)\mu_{k,t}' \ \  (\forall t\neq k).$$
Let~$m=0$. It follows that~$\mu_{k,k}'=\varkappa$ for all~$k\neq 0$; and~$\mu_{k,t}'=0$ for all~$t\notin \{k, 2k\}$.  Combining this with the above equality $2 (\frac{k}{2}-\ell m)\mu_{k+m,k+p}'=\nu_{m,p}(\frac{k}{2}-\ell p) (\forall m\neq p)$,
we deduce that~$\nu_{m,p}=0$ for all~$m\neq p$, and thus~$\mu_{k,t}'=0$ for all~$k\neq t$. For~$\ell\in \mathbb{Z}$, let~$0\neq k=-m$,  we deduce that~$\mu_{0,0}'=\mu_{k,k}'=\varkappa$; For~$\ell\notin \mathbb{Z}$, we have~$k\notin \mathbb{Z}$.  So we have~$\varphi_0(H_k)=\varkappa H_k$ for all~$k$.

Since~$2\mu_{k+m,k+m}'=\varkappa +\nu_{m,m}$ for all~$k,m$ satisfying~$k\neq 2\ell m$, we have~$\nu_{m,m}=\varkappa $ and thus~$\varphi_0(G_m)
=\sum_{k}\mu_{m,k}H_k+\varkappa G_m$ for all~$m$.

So we have
{\setlength{\tabcolsep}{1pt}\begin{longtable}{lcl}
$2(\frac{m}{2}-n)(\sum_{k}\mu_{m+n,k}H_k$&$+$&$\varkappa G_{m+n})=2(\frac{m}{2}-n)\varphi_0(G_{m+n})$\\
&$=$&$2\varphi_0([L_m, G_n])
=\varkappa(\frac{m}{2}-n)G_{m+n}+[L_m,\sum_{k}\mu_{n,k}H_k+\varkappa G_{n}]$\\
&$=$&$2\varkappa(\frac{m}{2}-n)G_{m+n}
+\sum_{k}\mu_{n,k}(\frac{2\ell-1}{2}m-k)H_{m+k}.$
\end{longtable}}
Let~$m=2n$. Then we obtain~$\mu_{n,k}=0$ if~$(2\ell-1)n\neq k $.  In particular, if~$\ell\notin \mathbb{N}$, then~$(2\ell-1)n\neq k $ and thus~$\mu_{n,k}=0$ for all~$n,k$.  If~$\ell\in \mathbb{Z}$, then we set~$m=0$ and thus
$$(-2n)\Big(\sum_{k}\mu_{n,k}H_k+\varkappa G_n\Big)=2\varkappa(-n)G_{n}
+\sum_{k}\mu_{n,k}(-k)H_{k}.$$
So~$(2n-k)\mu_{n,k}=0$, in particular, $\mu_{n,k}=0$ for all~$k\neq 2n$ and thus
$\varphi_0(G_n)
=\mu_{n,2n}H_{2n}+\varkappa G_n$ for all~$n$. But then the above formula becomes
\begin{center}
$2(\frac{m}{2}-n)\Big(\mu_{m+n,2(m+n)}H_{2(m+n)}+\varkappa G_{m+n}\Big)
=2\varkappa(\frac{m}{2}-n)G_{m+n}
+\mu_{n,2n}(\frac{2\ell-1}{2}m-2n)H_{m+2n}.$
\end{center}
Let~$m\neq 0$ and let~$\frac{2\ell-1}{2}m\neq 2n$. We have~$\mu_{n,2n}=0$.
So~$\varphi_0(G_n)=\varkappa G_n$ and thus~$\varphi_0$ is trivial.

As a conclusion, all even $\frac{1}{2}$-superderivations of $\mathfrak{gca}(\ell)$ are trivial if $\ell\neq1$, and they have the form $\varkappa\,{\rm id}+\mu\Psi$ if $\ell=1$.

It is known that the supercommutator of a $\frac{1}{2}$-superderivation and one superderivation gives a new $\frac{1}{2}$-superderivation.
Now, let $\mathbf{ad}_x$ be an inner odd derivation of $\mathfrak{gca}(\ell)$. Then $\llbracket \varphi_1, \mathbf{ad}_x\rrbracket_s$ is an even $\frac{1}{2}$-superderivation of $\mathfrak{gca}(\ell)$; assume $\llbracket \varphi_1, \mathbf{ad}_x\rrbracket_s=\alpha_{x} {\rm id}$ (if $\ell=1$, then $\llbracket \varphi_1, \mathbf{ad}_x\rrbracket_s=\alpha_{x} {\rm id}+\mu_x\Psi$, and the computation below remains valid: the terms $\mu_x\Psi(L_m)=\mu_xP_m$ are treated in the same way as the terms $\alpha_xL_m$, and they give $\mu_x=0$).

Suppose \begin{center}$\varphi_1(G_n)=\sum_{m}\mu_{n,m}'L_m+\sum_{t}\nu_{n,t}'P_t
+\rho_{n,1}'c_1+\rho_{n,2}'c_2$ and~$\varphi_{1}(L_m)=\sum_{p}\alpha_{m,p}G_p+\sum_{t}\beta_{m,t}H_t$.
\end{center}
Then we have
{\setlength{\tabcolsep}{1pt}\begin{longtable}{lcl}
  $\alpha_{G_n} L_m$&$=$&$\llbracket\varphi_1, \mathbf{ad}_{G_n}\rrbracket_s(L_m)
  =\varphi_1[G_n,L_m]+[G_n,\varphi_1(L_m)]$\\
  &$ =$&$-(\frac{m}{2}-n)(\sum_{p}\mu_{n+m,p}'L_p+\sum_{t}\nu_{n+m,t}'P_t
+\rho_{n+m,1}'c_1+\rho_{n+m,2}'c_2)$\\
&&$+\sum_{p}\alpha_{m,p}[G_n,G_p]+\sum_{t}\beta_{m,t}[G_n, H_t]$\\
 &$=$&$-(\frac{m}{2}-n)(\sum_{p}\mu_{n+m,p}'L_p+\sum_{t}\nu_{n+m,t}'P_t
+\rho_{n+m,1}'c_1+\rho_{n+m,2}'c_2)$\\
&&$+\sum_{p}\alpha_{m,p}(2L_{n+p}+c_1(4n^2-1)\delta_{n+p,0})
+\sum_{t}\beta_{m,t}(2P_{n+t}+c_2(4n^2-1)\delta_{n+t,0}\delta_{\ell,1})$
\end{longtable}}
For~$m=2n$, we have~$\alpha_{2n,p}=0=\beta_{2n,t}$ for all~$n,p,t$ satisfying~$n\neq p$; and~$\alpha_{G_n}=2\alpha_{2n,n}$. In particular, $\varphi_1(L_{2n})=\alpha_{2n,n}G_n$ for all~$n$.    Let~$m=2q\neq \pm 2n$. Then we have
\[\adjustbox{max width=\linewidth}{$\displaystyle \alpha_{G_n} L_{2q}=2\alpha_{2n,n} L_{2q}=-(q-n)\Big(\sum_{p}\mu_{n+2q,p}'L_p+\sum_{t}\nu_{n+2q,t}'P_t
+\rho_{n+2q,1}'c_1+\rho_{n+2q,2}'c_2\Big)+\alpha_{2q,q}2L_{n+q}$}\]
It follows that
$\varphi_1(G_n)=\sum_{m}\mu_{n,m}'L_m$ and~$\mu_{n+2q,p}'=0$ if~$p\neq 2q $ and~$p\neq n+q$.  Since~$n,q$ are arbitrary, it follows that~$\mu_{n,p}'=0$ for all~$n,p$. And thus~$\varphi_1(G_n)=0$ for all~$n$.

Then we have
\begin{longtable}{lcl}
 $0$&$=$&$2(\frac{m}{2}-n)\varphi_1 (G_{m+n})
=2\varphi_1[L_m,G_n]-[L_m, \varphi_1(G_n)]$\\
 & $=$&$[\varphi_1(L_m),G_n]=\sum_{p}\alpha_{m,p}[G_p,G_n]
 +\sum_{t}\beta_{m,t}[H_t,G_n]$\\
&$=$&$\sum_{p}\alpha_{m,p}\Big(2L_{p+n}+c_1(4p^2-1)\delta_{p+n,0}\Big)
+\sum_{t}\beta_{m,t}\Big(2P_{t+n}+c_2(4n^2-1)\delta_{n+t,0}\delta_{\ell,1}\Big).$
\end{longtable}
It follows that~$\alpha_{m,p}=0=\beta_{m,t}$ for all~$m,p,t$. In particular, we have~$\varphi_1(L_m)=0$ for all~$m$.
By applying~$\varphi_1$ to   the relation involving~$[G_1,G_{-1}]$, we deduce that~$\varphi_1(c_1)=0$.

Suppose~$\varphi_1(H_k)=\sum_{p}\mu_{k,p}L_p+\sum_{t}\nu_{k,t}P_t
+\rho_{k,1}c_1+\rho_{k,2}c_2$.
Then we have
\begin{longtable}{lcl}
 $\alpha_{H_k} L_0$&$=$&$ \llbracket\varphi_1, \mathbf{ad}_{H_k}\rrbracket_s(L_0)= \varphi_1[H_k, L_0]+[H_k, \varphi_1(L_0)]$\\
& $=$&$k\varphi_1(H_{k})
=k(\sum_{p}\mu_{k,p}L_p+\sum_{t}\nu_{k,t}P_t
+\rho_{k,1}c_1+\rho_{k,2}c_2).$
\end{longtable}
It follows that~$\alpha_{H_0}=0$ (if $\ell\in \mathbb{Z}$) and~$\varphi_1(H_k)=\mu_{k,0}L_0$ for all~$k\neq 0$.
So we have
\begin{longtable}{lcl}
  $\alpha_{H_k} L_m$ &$=$&$ \llbracket\varphi_1, \mathbf{ad}_{H_k}\rrbracket_s(L_m)= \varphi_1[H_k, L_m]$\\
&$=$&$-\Big(\frac{2\ell-1}{2}m-k\Big)\varphi_1(H_{m+k})
=-\Big(\frac{2\ell-1}{2}m-k\Big)\mu_{m+k,0}L_0.$
\end{longtable}
For~$m\notin\{0, -k\}$, we deduce that~$\alpha_{H_k}=0$ and thus~$\varphi_1(H_k)=0$ for all~$k\neq 0$. If~$\ell\notin \mathbb{N}$, we have~$\varphi_1(H_k)=0$ for all~$k$.  If~$\ell\in \mathbb{N}$, then we have~$\frac{2\ell-1}{2}m+m\neq 0$ for all nonzero integer~$m$. Let~$k=-m\neq 0$. This is possible because~$\ell\in \mathbb{N}$. Since
\begin{center}
$2\varphi\Big((\frac{2\ell-1}{2}m+m)H_{0}\Big)=2\varphi_1[L_m,H_m]
=[\varphi_1 (L_m),H_{-m}]+[L_m,\varphi_1 (H_{-m})]=0,$
\end{center}
we obtain~$\varphi_1(H_0)=0$  and thus~$\varphi_1(H_k)=0$ for all~$k$.

Suppose~$\varphi_{1}(P_k)=\sum_{m}\alpha_{k,m}'G_m+\sum_{t}\beta_{k,t}'H_t$.
Then we have
\begin{longtable}{lcl}
  $0$&$=$&$2\varphi_1[P_k,G_n]-[P_k,\varphi_1(G_n)]
=[\varphi_1(P_k), G_n]$\\
&$=$&$\sum_{m}\alpha_{k,m}'[G_m,G_n]+\sum_{t}\beta_{k,t}'[H_t,G_n].$
\end{longtable}
It follows that~$\alpha_{k,m}'=0=\beta_{k,t}'$ for all~$k,m,t$.

Finally, by applying~$\varphi_1$ to the relation involving $[G_1, H_{-1}]$ (if~$\ell=1$), we have~$\varphi_1(c_2)=0$.
Hence, $\Delta(\mathfrak{gca}(\ell))$ is trivial for $\ell\neq1$, and there are no nontrivial transposed Poisson superalgebra structures defined on $\mathfrak{gca}(\ell)$ by Theorem \ref{princth}.
If $\ell=1$, then $\Delta(\mathfrak{gca}(1))=\langle{\rm id},\Psi\rangle$, and all $\frac12$-superderivations are even.
Let $(\mathfrak{gca}(1),\cdot,[\cdot,\cdot])$ be a transposed Poisson superalgebra. Then the multiplication by an odd element is an odd $\frac12$-superderivation, i.e., it is zero, and the multiplication by an even element $x$ is $a(x){\rm id}+b(x)\Psi$.
For an even $x$ and $n\in\mathbb{Z}$, the supercommutativity gives $0=G_n\cdot x=x\cdot G_n=a(x)G_n+b(x)H_n$; hence, $a(x)=b(x)=0$.
Thus, the multiplication is zero, and there are no nontrivial transposed Poisson superalgebra structures defined on $\mathfrak{gca}(1)$.
\end{proof}

\section{$\frac{1}{2}$-derivations of some Lie algebras}

\subsection{$\frac{1}{2}$-derivations and transposed Poisson structures of solvable Lie algebras}
It is known  that each finite-dimensional nilpotent Lie algebra has a non-trivial transposed Poisson structure ($\frac{1}{2}$-derivations, $\frac{1}{2}$-biderivations) \cite[Theorem 16]{bfk22}.
  These  results are motivating the question of the existence of non-trivial $\frac{1}{2}$-derivations of solvable Lie algebras, which will be answered in the present subsection.

\begin{lemma}
\label{decom}
Let~$\mL$ be a decomposable Lie algebra, (namely, $\mL$ is the direct sum of two nonzero ideals). Then~$\mL$ has non-trivial $\frac{1}{2}$-derivations.
\end{lemma}
\begin{proof}
  Assume that~$\mL=I\oplus J$. Then for all~$x=y+z\in \mL$, where~$y$ lies in~$I$ and~$z$ lies in~$J$, we define~$\varphi(x)=z$. Clearly~$\varphi$ is a non-trivial $\frac{1}{2}$-derivation of~$\mL$.
\end{proof}

In the light of Lemma~\ref{decom}, we shall study non-abelian indecomposable Lie algebras~$\mL$.  Moreover, we shall focus on Lie algebras~$\mL$ such that~$\mL\neq [\mL,\mL]$. For all subspaces~$V,W\subseteq \mL$, define~$\mathsf{Ann}_{V}(W)=\{\xx\in V\mid [\xx,W]=0\}$.
\begin{lemma}\label{ann}
If $\mL$ is a Lie algebra such that $\dim\mL>1$, $\mL\neq [\mL,\mL]$, and $\mathsf{Ann}_{\mL}(\mL)\neq 0$, then $\mL$ has nontrivial $\frac{1}{2}$-derivations.
\end{lemma}
\begin{proof}
    If $\mL$ is abelian, then each linear map is a $\frac12$-derivation, and, since $\dim\mL>1$, there are nontrivial ones. Let $\mL$ be non-abelian.
  If $[\mL,\mL]\cap \mathsf{Ann}_{\mL}(\mL)=0$, then there exists a subspace $V$ of $\mL$ such that $\mL=V\oplus [\mL,\mL]\oplus \mathsf{Ann}_{\mL}(\mL)$ as vector spaces, and $V\oplus [\mL,\mL]$ and $\mathsf{Ann}_{\mL}(\mL)$ are two nonzero ideals of $\mL$. By Lemma \ref{decom}, $\mL$ has nontrivial $\frac{1}{2}$-derivations.

 If~$[\mL,\mL]\cap \mathsf{Ann}_{\mL}(\mL)\neq 0$, then there exists a nonzero element~$x_1\in [\mL,\mL]\cap \mathsf{Ann}_{\mL}(\mL)$. Assume that~$\mL=[\mL,\mL]\oplus W$ as vector spaces, where~$W$ is a nonzero subspace of~$\mL$. Now we extend~$x_1$ into a linear basis~$X$ of~$[\mL, \mL]$ and assume that~$Y$ is a linear basis of~$W$. Then we define an endomorphism~$\varphi$ of~$\mL$ by~$\varphi(x)=0$ for all~$x\in X$ and~$\varphi(y)=x_1$ for all~$y\in Y$. It follows that~$\varphi(\mL)\subseteq \mathsf{Ann}_{\mL}(\mL)$ and~$\varphi([\mL,\mL])=0$.  So we deduce that~$\varphi$ is a non-trivial $\frac{1}{2}$-derivation.
  \end{proof}

We shall prove that every nonabelian solvable finite dimensional Lie algebra over an algebraic closed field of zero characteristic has non-trivial $\frac{1}{2}$-derivations. Before proceeding to the proof, we first recall a well-known result on solvable Lie algebras, for instance, see~\cite[page 15]{hum72}.
\begin{theorem}[see \cite{hum72}]\label{com-v}
  Let~$\mL$ be a solvable subalgebra of~$\mathfrak{gl}(V)$, where~$V$ is a finite dimensional nonzero vector space over an algebraic closed field of zero characteristic and~$\mathfrak{gl}(V)$ is the Lie algebra consisting of all the endomorphisms of~$V$.  Then~$V$ contains a common eigenvector for all the endomorphisms in~$\mL$.
  \end{theorem}

\begin{theorem}\label{nontrslv}
  Let~$\mL$ be a  solvable finite dimensional Lie algebra over an algebraic closed field of zero characteristic such that   $\mathsf{dim}(\mL)>1$.  Then~$\mL$ has non-trivial $\frac{1}{2}$-derivations.
\end{theorem}

\begin{proof}
Since~$\mL$ is solvable and finite dimensional, it is well-known that~$\mL\neq [\mL,\mL]$ and~$[\mL,\mL]$ is nilpotent, (for instance, see~\cite[page 16]{hum72}). If~$\mathsf{Ann}_{\mL}(\mL)\neq 0$, then by Lemma~\ref{ann},  $\mL$ has non-trivial~$\frac{1}{2}$-derivations.

Now we assume that~$\mathsf{Ann}_{\mL}(\mL)=0$ and denote~$\mathsf{Ann}_{[\mL,\mL]}([\mL,\mL])$ by~$W$. Since~$[\mL,\mL]$ is nilpotent, it follows that~$W\neq 0$. We claim that~$W$ is an ideal of~$\mL$: For all~$w\in W$ and~$x,y,z\in \mL$,  we have~$[w,x]\in [\mL,\mL]$ and
\begin{center}
    $[[w,x],[y,z]]=[[w,[y,z]],x]+[w,[x,[y,z]]]=0$, namely, we have~$[w,x]\in W$.
    \end{center}
It follows that $\mL$ acts on~$W$ via the adjoint representation, namely, for all~$x\in \mL$ and~$w\in W$, we have~$x.w=\mathsf{ad }_x (w)=[x,w]$.

 Now we identify $\mathsf{ad }_x$ as an endomorphism of~$W$.  Then $\mathsf{ad}_ \mL=\{\mathsf{ad}_x \mid x\in \mL\}$ is a finite dimensional solvable subalgebra of~$\mathfrak{gl}(W)$. By Theorem~\ref{com-v}, there exists a nonzero element~$w_0\in W$ such that~$\mathsf{ad}_x (w_0)=\lambda_{x} w_0$ for every~$x\in \mL$, where each~$\lambda_x$ is an element in the underlying field depending on~$x$.  Moreover, we note that there exists an $x\in \mL$ such that~$[x,w_0]\neq 0$  since~$\mathsf{Ann}_{\mL}(\mL)= 0$.  It follows that
 $\varphi:\mL\rightarrow \mL, x\mapsto [w_0,x]$ \ (for all $x\in \mL$)  is a nonzero endomorphism of~$\mL$. Finally, since for all~$x,y\in \mL$, we have
\begin{center}$\varphi([x,y])=0=\frac{1}{2}[w_0,[x,y]]=\frac{1}{2} ([[w_0,x],y]+[x,[w_0,y]])=\frac{1}{2}     ([\varphi(x),y]+[x,\varphi(y)]).$\end{center}
So~$\varphi$ is a nonzero $\frac{1}{2}$-derivation of~$\mL$.
\end{proof}

	Let us recall the definition of ${\rm Hom}$-structures on Lie algebras.
	\begin{definition}
		Let $({\mathfrak L}, [\cdot,\cdot])$ be a Lie algebra and $\varphi$ be a linear map.
		Then $({\mathfrak L}, [\cdot,\cdot], \varphi)$ is a ${\rm Hom}$-Lie structure on $({\mathfrak L}, [\cdot,\cdot])$ if 
		\[
		[\varphi(x),[y,z]]+[\varphi(y),[z,x]]+[\varphi(z),[x,y]]=0.
		\]
	\end{definition}

Filippov proved that each nonzero $\delta$-derivation ($\delta\neq0,1$) of a Lie algebra, 
 gives a non-trivial ${\rm Hom}$-Lie algebra structure \cite[Theorem 1]{fil1}.
 Hence, by Theorem \ref{nontrslv}, we have the following corollary.

 \begin{corollary}
  Let~$\mL$ be a  solvable finite dimensional Lie algebra over an algebraic closed field of zero characteristic such that   $\mathsf{dim}(\mL)>1$.  
  Then~$\mL$   admits a non-trivial ${\rm Hom}$-Lie algebra structure.
 \end{corollary}

\begin{theorem}\label{nilpotent}
    Let $\mathfrak{L}$ be a finite-dimensional non-abelian solvable Lie algebra over an algebraically closed field of characteristic $0$.
Then $\mathfrak{L}$ admits a nontrivial transposed Poisson structure.
\end{theorem}

\begin{proof}
Since $\mL$ is non-abelian, its Lie product is nonzero. If~$\mathsf{Ann}_{\mL}(\mL)\neq  0$ and~$\mL=V\oplus [\mL,\mL]\oplus \mathsf{Ann}_{\mL}(\mL)$  as vector spaces. Let~$Y=\{y_1,\wdots y_m \}$ be a linear basis of~$V\oplus [\mL,\mL]$ and let~$X=\{x_1,\dots, x_n\}$ be a linear basis of~$\mathsf{Ann}_{\mL}(\mL)$. Define~$\cdot $ on $\mL$ by~$x_i\cdot x_j=x_1$, and define $z_1\cdot z_2=0$ if~$\{z_1,z_2\}\subseteq X\cup Y$ and~$\{z_1,z_2\}\nsubseteqq X$.  Clearly, $(\mL, \cdot, [\cdot, \cdot])$ is a nontrivial transposed Poisson  structure.

If~$[\mL,\mL]\cap \mathsf{Ann}_{\mL}(\mL)\neq 0$, then there exists a nonzero element~$y_1\in [\mL,\mL]\cap \mathsf{Ann}_{\mL}(\mL)$. So we may assume that~$\mL=W\oplus [\mL,\mL]$ as vector spaces. Suppose that $Y=\{y_1,\ldots, y_m \}$ is a linear basis of~$ [\mL,\mL]$ and~$X=\{x_1,\dots, x_n\}$ is a linear basis of~$W$. Define~$x_i\cdot x_j=y_1$, and define $z_1\cdot z_2=0$ if~$\{z_1,z_2\}\subseteq X\cup Y$ and~$\{z_1,z_2\}\nsubseteqq X$. Clearly, $(\mL, \cdot, [\cdot, \cdot])$ is a nontrivial transposed Poisson  structure.

  If~$\mathsf{Ann}_{\mL}(\mL)= 0$, then, with the notation of the proof of Theorem~\ref{nontrslv}, for all $x,y\in \mL$, we define~$x\cdot y=[[w_0,x],y]$. Let~$a$ be an element in $\mL$ such that~$[w_0,a]=w_0$. Then we have~$a\cdot a=w_0\neq 0$. Moreover, for all~$x,y,z\in \mL$, since~$w_0\in \mathsf{Ann}_{[\mL,\mL]}([\mL,\mL])$, we have
$$x\cdot y=[[w_0,x],y]=[[w_0,y],x]+[w_0,[x,y]]=[[w_0,y],x]=y\cdot x$$
and
$$(x\cdot y)\cdot z=[[w_0,x],y] \cdot z=[[w_0, [[w_0,x],y]],z]=0=(y\cdot z)\cdot x =x\cdot (y\cdot z).$$
So~$(\mL, \cdot)$ is an associative commutative algebra (of nilpotent index 3).
Moreover, since $\mathsf{Ann}_{[\mL,\mL]}([\mL,\mL]) $ is an ideal of~$\mL$,  we have
$$x\cdot [y,z]=[y,z]\cdot x=[[w_0, [y,z]], x]=0$$
and
$$[x\cdot y, z]+[y,x\cdot z]=[[[w_0,x],y],z]+[y, [[w_0,x],z]]
=[[[w_0,x],y],z]- [[w_0,x],z],y]=0.$$
So $(\mL,\cdot, [\cdot, \cdot])$ is a non-trivial transposed Poisson  structure. The proof is complete.
\end{proof}

\subsection{$\frac{1}{2}$-derivations and central extensions}
We also note that if~$\mL/\mathsf{Ann}_{\mL}(\mL)$ has only trivial $\frac{1}{2}$-derivations, then every $\frac{1}{2}$-derivation of~$\mL$ is in the centroid of~$\mL$.
\begin{lemma}
 If~$\mL/\mathsf{Ann}_{\mL}(\mL)$ has only trivial $\frac{1}{2}$-derivations, then for all $\frac{1}{2}$-derivations~$\varphi$ of~$\mL$, for all~$x,y\in \mL$, we have
 \begin{center}$\varphi([x,y])=\alpha[x,y]=[\varphi(x),y]=[x,\varphi(y)]$
 \end{center} for some element~$\alpha$ from the underlying field.
\end{lemma}
\begin{proof}
  For all~$x\in \mathsf{Ann}_{\mL}(\mL) $ and~$y\in \mL$, we have
  $$[\varphi(x),y]=2\varphi([x,y])-[x,\varphi(y)]=0.$$
So we obtain that~$\varphi(\mathsf{Ann}_{\mL}(\mL))\subseteq \mathsf{Ann}_{\mL}(\mL)$.  Therefore, $\varphi$ induces an endomorphism~$\overline{\varphi}$ of~$\mL/\mathsf{Ann}_{\mL}(\mL)$ by~$\overline{\varphi}(x+\mathsf{Ann}_{\mL}(\mL))=\varphi(x)+\mathsf{Ann}_{\mL}(\mL)$ for every~$x\in \mL$. Moreover, since~$\mathsf{Ann}_{\mL}(\mL)$ is an ideal of~$\mL$, for all~$x,y\in \mL$, we have
\begin{longtable}{lcl}
 $\overline{\varphi}\Big([x+\mathsf{Ann}_{\mL}(\mL),y+\mathsf{Ann}_{\mL}(\mL)]\Big)$
&$=$&$\overline{\varphi}([x,y]+\mathsf{Ann}_{\mL}(\mL))=\varphi([x,y])+\mathsf{Ann}_{\mL}(\mL)$\\
&$=$&$ \frac{1}{2} \Big([\varphi(x),y]+[x,\varphi(y)]\Big)+\mathsf{Ann}_{\mL}(\mL).$
\end{longtable}
It follows that~$\overline{\varphi}$ is a $\frac{1}{2}$-derivation of~$\mL/\mathsf{Ann}_{\mL}(\mL)$. By assumption, there exists an element~$\alpha$ of the underlying field such that~$\varphi(x)+\mathsf{Ann}_{\mL}(\mL)=\overline{\varphi}(x+\mathsf{Ann}_{\mL}(\mL))=\alpha x+\mathsf{Ann}_{\mL}(\mL)$ for every~$x\in \mL$.
So we have
\begin{center}$\varphi(x)-\alpha(x) \in \mathsf{Ann}_{\mL}(\mL)$.
\end{center} Therefore, for all~$x,y\in \mL$, we have
\begin{longtable}{lcl}
$ \varphi([x,y])$&$=$&$\frac{1}{2} \Big([\varphi(x),y]+[x,\varphi(y)]\Big)
=\frac{1}{2} \Big([\alpha x,y]+[x,\alpha y]\Big)$\\
&$=$&$\alpha([x,y])=[\varphi(x),y]=[x,\varphi(y)].$
\end{longtable}
 The proof is complete.
\end{proof}






\appendix
\section*{Appendix A. Differences from the published version}
\addcontentsline{toc}{section}{Appendix A. Differences from the published version}

\noindent\textbf{Theorems \ref{galthm}, \ref{Gthm} and \ref{gcathm}:} the proofs have been corrected. 
The Lie algebra $\mathfrak{gal}(3)$ has a nontrivial $\frac12$-derivation, so that the published argument fails for $d=3$; the transposed Poisson structures on $\mathfrak{gal}(3)$ are nevertheless trivial.
The superalgebra $\mathfrak{gca}(1)$ has a nontrivial $\frac12$-superderivation, so that the published argument fails for $\ell=1$; the transposed Poisson superalgebra structures on $\mathfrak{gca}(1)$ are nevertheless trivial.

\end{document}